\documentclass[12pt]{article}%
\usepackage{amsmath}
\usepackage{amsfonts}
\usepackage{amssymb}
\usepackage{graphicx}%
\providecommand{\U}[1]{\protect\rule{.1in}{.1in}}
\newtheorem{theorem}{Theorem}

\newtheorem{corollary}[theorem]{Corollary}

\newtheorem{example}[theorem]{Example}

\newtheorem{lemma}[theorem]{Lemma}

\newenvironment{proof}[1][Proof]{\noindent\textbf{#1.} }{\ \rule{0.5em}{0.5em}}
\begin{document}

\title{Efficiency of the singular vector of a reciprocal matrix and comparison to the
Perron vector}
\author{Susana Furtado\thanks{Email: sbf@fep.up.pt The work of this author was
supported by FCT- Funda\c{c}\~{a}o para a Ci\^{e}ncia e Tecnologia, under
project UIDB/04721/2020.} \thanks{Corresponding author.}\\CEAFEL and Faculdade de Economia \\Universidade do Porto\\Rua Dr. Roberto Frias\\4200-464 Porto, Portugal
\and Charles R. Johnson \thanks{Email: crjohn@wm.edu. }\\Department of Mathematics\\College of William and Mary\\Williamsburg, VA 23187-8795}
\maketitle

\begin{abstract}
In models using pair-wise (ratio) comparisons among alternatives (reciprocal
matrices $A$) to deduce a cardinal ranking vector, the right Perron
eigenvector was traditionally used, though several other options have emerged.
We propose and motivate another alternative, the left singular vector (Perron
eigenvector of $AA^{T}$). Theory is developed. We show that for reciprocal
matrices obtained from consistent matrices by modifying one column, and the
corresponding row, the cone generated by the columns is efficient, implying
that the Perron vector and the left singular vector of such matrices are
always efficient. In addition, the two vectors are compared empirically for
random matrices. With regard to efficiency, the left singular vector
consistently outperforms the right Perron eigenvector, though both perform
well for large numbers of alternatives.

\end{abstract}



\bigskip

\section{Introduction}

An $n$-by-$n$ positive matrix $A=[a_{ij}]$ is called \emph{reciprocal} if
$a_{ji}=\frac{1}{a_{ij}}$ for every pair $1\leq i,j\leq n$. The diagonal
entries are all $1$. We denote the set of such matrices by $\mathcal{PC}_{n}.$
Reciprocal matrices $A\in\mathcal{PC}_{n}$ arise as an array of pair-wise
(ratio) comparisons among $n$ alternatives. In this context, typically a
cardinal ranking vector is to be extracted from $A$ for decision problems
about business alternatives
\cite{anh,choo,dij,fichtner86,golany,is,Kula,zeleny}; and other applications,
in voting theory and currency exchange rates, are emerging. Historically, the
(right) Perron eigenvector of $A\ $has been used for this vector of weights,
as suggested in \cite{saaty1977,Saaty}. However, many natural concerns about
this choice have emerged (see, for example, \cite{johns}). Here, we suggest an
entirely new, but natural alternative, the Perron eigenvector of $AA^{T},$
sometimes called the \emph{left singular vector} of $A,$ as it is the left
singular vector of $A$ associated with the largest singular value.

The comparisons embodied in a reciprocal $A$ are consistent ($a_{ik}%
=a_{ij}a_{jk}$ for all $1\leq i,j,k\leq n$) if and only if, for some positive
$n$-vector $w,$ $A=ww^{(-T)}$ (in which $w^{(-T)}$ denotes the entry-wise
inverse of the transpose of $w$). In this case $w$ is the only natural weight
vector (up to a factor of scale), but consistency is unlikely in independent
pair-wise comparisons. One agreeable criterion to be a weight vector $w$ is
that the consistent matrix built from $w$, $W=ww^{(-T)}$, be a Pareto optimal
approximation to $A.$ This means that, if $\left\vert A-V\right\vert
\leq\left\vert A-W\right\vert ,$ absolute value entry-wise, in which
$V=vv^{(-T)}$ for a positive vector $v$, then $v$ is proportional to $w.$ In
this literature, such a positive vector $w$ is called \emph{efficient} (for
$A$). However, when $A$ is inconsistent, there will be infinitely many,
projectively distinct efficient vectors. Denote by $\mathcal{E}(A)$ the
collection of all vectors efficient for $A.$ It is now known that every column
of $A$ is efficient \cite{FJ1}, the entry-wise geometric convex hull of the
columns \cite{FJ2} and, in particular, the simple entry-wise geometric mean of
all the columns \cite{blanq2006}, are efficient. However, the Perron vector is
not always efficient \cite{bozoki2014,FJ3}. A graph characterization of
efficiency was given in \cite{blanq2006} (see also \cite{FJ2}). Several other
developments on efficiency have appeared (see, for example,
\cite{p6,p2,baj,european,CFF,FerFur,Fu22}).

Since the definition of efficiency for $A\in\mathcal{PC}_{n}$ involves rank
$1$ approximation to $A$, and the best rank $1$ approximation (in the
Frobenius norm) is the normalized outer product of the right and left singular
vectors of $A$ (associated with the spectral norm) \cite{HJ2}$,$ it seems
natural to test one of these for efficiency. Of course, the best rank 1
approximation is not usually reciprocal, let alone consistent.

Here, we primarily compare the left singular vector of $A\in\mathcal{PC}_{n}$
(i.e. the Perron vector of $AA^{T}$) to the Perron eigenvector under the lens
of efficiency. In the next section we amass the technical background that we
need. Then, in Section \ref{scone}, we develop some additional technology for
understanding efficiency. In particular, it is noted that, in the event that
$\mathcal{E}(A)$ is convex, both the singular vector and the Perron vector are
efficient. In Section \ref{s3} we give a unified proof that, among column
perturbed consistent matrices, both vectors are always efficient. Convexity of
$\mathcal{E}(A)$ occurs when $A$ lies in the class of simple perturbed
consistent matrices introduced in \cite{p6} (see also \cite{CFF}), but
unfortunately it is not yet known in general for which $A$ $\mathcal{E}(A)$ is
convex. In Section \ref{s2}, we accumulate considerable information comparing
the two vectors, via simulation. The singular vector performs better than the
Perron eigenvector in all dimensions $n>3$, both are more frequently efficient
the higher the dimension (approaching efficiency with probability $1$), and
more detailed comparison is presented. We give some final remarks in Section
\ref{s6}.

\section{Background\label{s22}}

The set $\mathcal{PC}_{n}$ is closed under monomial similarity, that is,
similarity via a product of a permutation matrix and a positive diagonal
matrix (monomial matrix) \cite{FJ1}. Fortunately, such transformation
interfaces with efficient vectors in a natural way.

\begin{lemma}
\label{lsim}\cite{FJ1} Suppose that $A\in\mathcal{PC}_{n}$ and $w\in
\mathcal{E}(A).$ If $S$ is an $n$-by-$n$ monomial matrix, then $Sw\in
\mathcal{E}(SAS^{-1})$.
\end{lemma}

\bigskip

Given $A=[a_{ij}]\in\mathcal{PC}_{n}$ and a positive vector
\[
w=\left[
\begin{array}
[c]{ccc}%
w_{1} & \cdots & w_{n}%
\end{array}
\right]  ^{T},
\]
define $G(A,w)$ as the directed graph (digraph) with vertex set $\{1,\ldots
,n\}$ and a directed edge $i\rightarrow j$ if and only if $w_{i}\geq
a_{ij}w_{j}$, $i\neq j.$

In \cite{blanq2006} the authors proved that the efficiency of $w$ can be
determined using $G(A,w)$. In \textrm{\cite{FJ2}} a shorter matricial proof of
this result has been given.

\begin{theorem}
\textrm{\cite{blanq2006}}\label{blanq} Let $A\in\mathcal{PC}_{n}$. A positive
$n$-vector $w$ is efficient for $A$ if and only if $G(A,w)$ is a strongly
connected digraph, that is, for all pairs of vertices $i,j,$ with $i\neq j,$
there is a directed path from $i$ to $j$ in $G(A,w)$.
\end{theorem}

We recall that the (right) Perron eigenvector of an $n$-by-$n$ positive matrix
$A$ is a (right) eigenvector of $A$ associated with the spectral radius
$\rho(A)$ of $A$ (called the Perron eigenvalue of $A$) \cite{HJ}. It is known
that, up to a constant factor, all its entries are positive and, with some
normalization, as having the sum of the entries equal to $1,$ it is unique.

If a positive matrix has rank $1,$ then any column is a Perron vector of the
matrix (all the other projectively distinct eigenvectors are associated with
the eigenvalue $0$). When $A$ is consistent, $A=ww^{(-T)},$ any column of $A,$
and of $AA^{T},$ is a multiple of $w.$ Thus, the Perron vectors of $A$ and
$AA^{T}$ (as any column) are projectively equal to $w$ and, thus, are
efficient vectors for $A.$ Here, we focus upon the efficiency of the Perron
eigenvector of $AA^{T}$ for a reciprocal matrix $A.$ Note that, since $AA^{T}$
is symmetric, its right and left Perron vectors coincide.

\bigskip

We denote by $\mathcal{C}(A)$ the (convex) cone generated by the columns of
$A$, that is, the set of nonzero, nonnegative linear combinations of the
columns of $A$. Note that, according to our definition, $0$ is not in
$\mathcal{C}(A)$. It is important to note that both the right Perron vector of
$A$ and the Perron vector of $AA^{T}$ lie in $\mathcal{C}(A).$

\begin{lemma}
\label{LPSV}For any positive $n$-by-$n$ matrix $A$ (in particular,
$A\in\mathcal{PC}_{n}$), both the right Perron eigenvector of $A$ and the
Perron eigenvector of $AA^{T}$ lie in $\mathcal{C}(A).$
\end{lemma}

\begin{proof}
If $Av=\rho(A)v,$ $v$ may be taken to be entry-wise positive. So $Av$ lies in
$\mathcal{C}(A),$ and since $v$ is proportional to $Av,$ via the positive
number $\rho(A),$ $v\in\mathcal{C}(A)$ also. If $AA^{T}u=\rho(AA^{T})u$, since
$AA^{T}$ is positive, $\rho(AA^{T})$ is positive and $u$ may be taken to be
positive. Now, $A^{T}u$ is positive so that $A(A^{T}u)\in\mathcal{C}(A),$ and,
as $u$ is proportional, via $\rho(AA^{T}),$ $u\in\mathcal{C}(A)$ also.
\end{proof}

\bigskip

Two simple lemmas will be helpful.

\begin{lemma}
\label{lconvex2}Let $A$ and $S$ be $n$-by-$n$ matrices, with $S$ monomial. If
$\mathcal{C}(A)$ is the cone generated by the columns of $A$ then
$S\mathcal{C}(A)$ is the cone generated by the columns of $SAS^{-1}.$
\end{lemma}

\begin{proof}
It is enough to show that $\mathcal{C}(SAS^{-1})\subseteq S\mathcal{C}(A),$ as
then applying this result to $A$ instead of $SAS^{-1}$ we get the other
inclusion. Let $S=DP,$ in which $D=\operatorname*{diag}(\lambda_{1}%
,\ldots,\lambda_{n})$ is a positive diagonal matrix and $P$ is a permutation
matrix (the proof is similar if $S=PD$). Let $a_{i}$ and $a_{i}^{\prime}$
denote the $i$th columns of $A$ and $A^{\prime}=SAS^{-1},$ respectively,
$i=1,\ldots,n.$ Let $x_{1},\ldots,x_{n}\geq0$ (not all $0$)$.$ Then,
\[
x_{1}a_{1}^{\prime}+\cdots+x_{n}a_{n}^{\prime}=S\left(  \frac{x_{1}}%
{\lambda_{1}}a_{\sigma(1)}+\cdots+\frac{x_{n}}{\lambda_{n}}a_{\sigma
(n)}\right)  \in S\mathcal{C}(A),
\]
where $\sigma$ is the permutation of $\{1,\ldots,n\}$ associated with $P^{T}.$
\end{proof}

\begin{lemma}
\label{lconvex}If $\mathcal{C}$ is a convex subset of $\mathbb{R}^{n},$ and
$S$ is an $n$-by-$n$ real matrix, then $S\mathcal{C}$ is also convex.
\end{lemma}

\begin{proof}
Let $Sw_{1},Sw_{2}\ $ be arbitrary vectors in $S\mathcal{C}.$ Let $\lambda
_{1},\lambda_{2}>0$ summing to $1.$ Since $w_{1},w_{2}\in\mathcal{C}$ and
$\mathcal{C}$ is convex, $\lambda_{1}w_{1}+\lambda_{2}w_{2}\in\mathcal{C}.$
Then, $\lambda_{1}Sw_{1}+\lambda_{2}Sw_{2}=S(\lambda_{1}w_{1}+\lambda_{2}%
w_{2})\in S\mathcal{C}.$\bigskip
\end{proof}

\begin{corollary}
\label{cconv}Let $A\in\mathcal{PC}_{n}$ and $S$ be an $n$-by-$n$ monomial
matrix. Then, $\mathcal{E}(A)$ is convex if and only if $\mathcal{E}%
(SAS^{-1})$ is convex.
\end{corollary}

\begin{proof}
By Lemma \ref{lsim}, $\mathcal{E}(SAS^{-1})=S\mathcal{E}(A).$ Thus, by Lemma
\ref{lconvex}, $\mathcal{E}(A)$ convex implies $\mathcal{E}(SAS^{-1})$ convex.
The converse follows similarly.
\end{proof}

\section{Efficiency and the cone generated by the columns of a reciprocal
matrix\label{scone}}

The cone generated by the columns of $A\in\mathcal{PC}_{n}$ may or may not be
contained in $\mathcal{E}(A)$. But it is quite helpful when it is.

\begin{theorem}
\label{t1}Let $A\in\mathcal{PC}_{n}$ and suppose that $\mathcal{C}%
(A)\mathcal{\subseteq E}(A)$. Then, both the Perron vector of $AA^{T}$ and the
right Perron vector of $A$ lie in $\mathcal{E}(A).$ Moreover, $\mathcal{C}%
(AA^{T})\subseteq\mathcal{E}(A).$
\end{theorem}

\begin{proof}
By Lemma \ref{LPSV}, both the right Perron eigenvector of $A$ and the Perron
eigenvector of $AA^{T}$ lie in $\mathcal{C}(A),$ so that the first claim
follows. The second claim is a consequence of the fact that $\mathcal{C}%
(AA^{T})\mathcal{\subseteq C}(A)$.
\end{proof}

\bigskip Recall that each column of $A\in\mathcal{PC}_{n}$ lies in
$\mathcal{E}(A)$. So, if $\mathcal{E}(A)$ is convex, $\mathcal{C}%
(A)\subseteq\mathcal{E}(A).$

\begin{corollary}
Let $A\in\mathcal{PC}_{n}$ and suppose that $\mathcal{E}(A)$ is convex. Then,
both the Perron vector of $AA^{T}$ and the right Perron vector of $A$ lie in
$\mathcal{E}(A).$
\end{corollary}

\bigskip

In Theorem \ref{t1}, we have identified a situation ($\mathcal{C}(A)$
contained in $\mathcal{E}(A)$) in which both the left singular vector and the
Perron vector of a reciprocal matrix are efficient for the matrix. If $n=3,$
this situation always occurs since $\mathcal{E}(A)$ is necessarily convex, as
we shall see. However, for any $n>3,$ there exists $A\in\mathcal{PC}_{n}$ such
that $\mathcal{C}(A)$ is not contained in $\mathcal{E}(A),$ in which case
$\mathcal{E}(A)$ is not convex. Any $A\ $with inefficient right Perron vector
establishes this (and this situation may occur for any $n\geq4$
\cite{bozoki2014,FJ3}). When $\mathcal{C}(A)$ is not contained in
$\mathcal{E}(A),$ either or both the Perron vector of $AA^{T}$ or the right
Perron vector of $A$ (or none) may lie in $\mathcal{E}(A)$.

\begin{example}
\label{ex01}Consider the reciprocal matrix%
\[
A=\left[
\begin{array}
[c]{ccccc}%
1 & 1.1742 & 0.5647 & 4.4912 & 0.3633\\
0.8516 & 1 & 1.4198 & 0.734\, & 0.8444\\
1.7709 & 0.7043 & 1 & 1.3358 & 1.7356\\
0.2227 & 1.3624 & 0.7486 & 1 & 5.4467\\
2.7525 & 1.1843 & 0.5762 & 0.1836 & 1
\end{array}
\allowbreak\right]  .
\]
The Perron eigenvectors of $A$ and $AA^{T}$ are, respectively,%
\[
w=\left[
\begin{array}
[c]{ccccc}%
1.4008 & 0.8134 & 1.1503 & 1.3488 & 1
\end{array}
\right]  ^{T}\text{ and }v=\left[
\begin{array}
[c]{ccccc}%
1.4667 & 0.8522 & 1.3142 & 2.2704 & 1
\end{array}
\right]  ^{T}.
\]
The vector $w$ is not efficient for $A,$ while $v$ is.
\end{example}

\begin{example}
\label{ex02}Consider the reciprocal matrix%
\[
A=\left[
\begin{array}
[c]{ccccc}%
1 & 1.6119 & 1.1855 & 2.1413 & 2.6124\\
0.6204 & 1 & 1.6338 & 3.7767 & 2.1376\\
0.8435 & 0.6121 & 1 & 0.2347 & 4.6488\\
0.4670 & 0.2648 & 4.2608 & 1 & 0.2462\\
0.3828 & 0.4678 & 0.2151 & 4.0617 & 1
\end{array}
\right]  .
\]
$\allowbreak$The Perron eigenvectors of $A$ and $AA^{T}$ are, respectively,%
\[
w=\left[
\begin{array}
[c]{ccccc}%
1.4842 & 1.5318 & 1.1940 & 1.0829 & 1
\end{array}
\right]  ^{T}\text{ and }v=\left[
\begin{array}
[c]{ccccc}%
1.1507 & 1.3620 & 1.0590 & 0.7746 & 1
\end{array}
\right]  ^{T}.
\]
The vector $w$ is efficient for $A,$ while $v$ is not.
\end{example}

\begin{example}
\label{ex03}Consider the reciprocal matrix%
\[
A=\left[
\begin{array}
[c]{ccccc}%
1 & 1.2767 & 0.5382 & 1.9522 & 0.2486\\
0.7833 & 1 & 0.5521 & 0.2793 & 2.2088\\
1.858 & 1.8113 & 1 & 1.54 & 1.1093\\
0.5122 & 3.5804 & 0.64935 & 1 & 0.7887\\
4.0225 & 0.4527 & 0.9015 & 1.2679 & 1
\end{array}
\right]  .
\]
$\allowbreak$The Perron eigenvectors of $A$ and $AA^{T}$ are, respectively,%
\[
w=\left[
\begin{array}
[c]{ccccc}%
0.6713 & 0.7041 & 0.9907 & 0.8591 & 1
\end{array}
\right]  ^{T}\text{ and }v=\left[
\begin{array}
[c]{ccccc}%
0.6131 & 0.5337 & 0.8907 & 0.8152 & 1
\end{array}
\right]  ^{T}.
\]
Both vectors $w$ and $v$ are not efficient for $A$.
\end{example}

Note that, from Theorem \ref{t1}, in Examples \ref{ex01}, \ref{ex02} and
\ref{ex03} $\mathcal{C}(A)$ is not contained in $\mathcal{E}(A).$

\begin{example}
\label{ex4}Consider the reciprocal matrix
\[
A=\left[
\begin{array}
[c]{ccccc}%
1 & 0.4245 & 0.21 & 1.0186 & 0.7814\\
2.3556 & 1 & 0.0168 & 5.6974 & 0.9107\\
4.7619 & 59.524 & 1 & 0.4913 & 0.1877\\
0.9817 & 0.1755 & 2.0354 & 1 & 1.5855\\
1.2798 & 1.0981 & 5.3277 & 0.6307 & 1
\end{array}
\right]  .
\]
The Perron eigenvectors of $A$ and $AA^{T}$ are, respectively,%
\[
w=\left[
\begin{array}
[c]{ccccc}%
0.1623 & 0.3658 & 1.9915 & 0.5080 & 1
\end{array}
\right]  ^{T}\text{ and }v=\left[
\begin{array}
[c]{ccccc}%
0.3988 & 0.9565 & 45.6075 & 0.2369 & 1
\end{array}
\right]  ^{T}.
\]
Both vectors $w$ and $v$ are efficient for $A$. However $w+v$ is not efficient
for $A,$ showing that $\mathcal{C}(A)$ is not contained in $\mathcal{E}(A).$
\end{example}

\section{Column perturbed consistent matrices\label{s3}}

$A\in\mathcal{PC}_{n}$ is called a \emph{column perturbed consistent matrix}
if $A$ has a consistent $(n-1)$-by-$(n-1)$ principal submatrix or, in other
words, if $A$ is a reciprocal matrix obtained from a consistent matrix by
modifying one column and the corresponding row. (Of course, if a column
perturbed consistent matrix has rank 1, it is consistent.) Since an
$(n-1)$-by-$(n-1)$ consistent matrix is positive diagonal similar to the
matrix $J_{n-1}$ of all $1$'s, and the columns of $A$ may be permuted to be in
any desired order, it follows that $A$ is monomially similar to a matrix of
the form
\begin{equation}
\left[
\begin{tabular}
[c]{c|c}%
$J_{n-1}$ & $%
\begin{array}
[c]{c}%
x_{1}\\
\vdots\\
x_{n-1}%
\end{array}
$\\\hline
$%
\begin{array}
[c]{ccc}%
\frac{1}{x_{1}} & \cdots & \frac{1}{x_{n-1}}%
\end{array}
$ & $1$%
\end{tabular}
\ \right]  , \label{AA}%
\end{equation}
with $x_{1}\geq\cdots\geq x_{n-1}>0.$

\begin{theorem}
\label{t6}Suppose that $A\in\mathcal{PC}_{n}$ is a column perturbed consistent
matrix. Then, $\mathcal{C}(A)\mathcal{\subseteq E}(A).$
\end{theorem}

\begin{proof}
Taking into account Lemmas \ref{lsim} and Lemma \ref{lconvex2}, we may assume
that $A$ has the form (\ref{AA}) with $x_{1}\geq\cdots\geq x_{n-1}>0.$
Consider the linear combination of the columns of $A$: $w=As,$ with
\[
s=\left[
\begin{array}
[c]{c}%
s_{1}\\
\vdots\\
s_{n-1}\\
s_{n}%
\end{array}
\right]  ,
\]
$s_{i}\geq0,$ $i=1,\ldots,n$ (not all $0$)$.$ We want to see that
$w\in\mathcal{E}(A),$ or, equivalently, by Theorem \ref{blanq}, that $G(A,w)$
is strongly connected. Let $ww^{(-T)}-A=[b_{ij}].$ We will show that
$b_{i,i+1}\geq0$ for $i=1,\ldots,n-1,$ and $b_{n,1}\geq0,$ implying that
\[
1\rightarrow2\rightarrow\cdots\rightarrow n\rightarrow1
\]
is a cycle in $G(A,w)$, and, thus, $G(A,w)$ is strongly connected.

Let $a_{i}^{T}$ denote the $i$th row of $A.$ For $i=1,\ldots,n-2,$ we have
\begin{align*}
b_{i,i+1}  &  =\left(  a_{i}^{T}s\right)  \left(  a_{i+1}^{T}s\right)
^{(-1)}-1\\
&  =\frac{s_{1}+\cdots+s_{n-1}+x_{i}s_{n}}{s_{1}+\cdots+s_{n-1}+x_{i+1}s_{n}%
}-1\\
&  =\frac{s_{n}(x_{i}-x_{i+1})}{s_{1}+\cdots+s_{n-1}+x_{i+1}s_{n}}\geq0.
\end{align*}

We also have
\begin{align*}
b_{n-1,n}  &  =\left(  a_{n-1}^{T}s\right)  \left(  a_{n}^{T}s\right)
^{(-1)}-x_{n-1}\\
&  =\frac{s_{1}+\cdots+s_{n-1}+x_{n-1}s_{n}}{\frac{s_{1}}{x_{1}}+\cdots
+\frac{s_{n-1}}{x_{n-1}}+s_{n}}-x_{n-1}\\
&  =\frac{s_{1}+\cdots+s_{n-1}-x_{n-1}(\frac{s_{1}}{x_{1}}+\cdots
+\frac{s_{n-1}}{x_{n-1}})}{\frac{s_{1}}{x_{1}}+\cdots+\frac{s_{n-1}}{x_{n-1}%
}+s_{n}}\\
&  =\frac{s_{1}(1-\frac{x_{n-1}}{x_{1}})+\cdots+s_{n-2}(1-\frac{x_{n-1}%
}{x_{n-2}})}{\frac{s_{1}}{x_{1}}+\cdots+\frac{s_{n-1}}{x_{n-1}}+s_{n}}\geq0.
\end{align*}
Finally,%
\begin{align*}
b_{n,1}  &  =\left(  a_{n}^{T}s\right)  \left(  a_{1}^{T}s\right)
^{(-1)}-\frac{1}{x_{1}}\\
&  =\frac{\frac{s_{1}}{x_{1}}+\cdots+\frac{s_{n-1}}{x_{n-1}}+s_{n}}%
{s_{1}+\cdots+s_{n-1}+x_{1}s_{n}}-\frac{1}{x_{1}}\\
&  =\frac{\frac{s_{1}}{x_{1}}+\cdots+\frac{s_{n-1}}{x_{n-1}}-\frac{1}{x_{1}%
}(s_{1}+\cdots+s_{n-1})}{s_{1}+\cdots+s_{n-1}+x_{1}s_{n}}\\
&  =\frac{s_{2}(\frac{1}{x_{2}}-\frac{1}{x_{1}})+\cdots+s_{n-1}(\frac
{1}{x_{n-1}}-\frac{1}{x_{1}})}{s_{1}+\cdots+s_{n-1}+x_{1}s_{n}}\geq0.
\end{align*}

\end{proof}

\begin{corollary}
If $A\in\mathcal{PC}_{n}$ is a column perturbed consistent matrix, then both
the Perron vector of $AA^{T}$ and the right Perron vector of $A$ lie in
$\mathcal{E}(A).$
\end{corollary}

\begin{proof}
Applying Theorems \ref{t1} and \ref{t6}, gives the desired conclusion.
\end{proof}

\bigskip

The efficiency of the Perron vector of a column perturbed consistent matrix
was obtained in \cite{FJ3}, but this gives a unified and simpler proof.

\bigskip It may happen that, for a column perturbed consistent matrix
$A\in\mathcal{PC}_{n},$ we have that $\mathcal{E}(A)$ is not convex. For this
to happen it must be that $n>3,$ as we shall see after the next example.

\begin{example}
Let%
\[
A=\left[
\begin{array}
[c]{cccccc}%
1 & 1 & 1 & 1 & 1 & 4\\
1 & 1 & 1 & 1 & 1 & 3\\
1 & 1 & 1 & 1 & 1 & 2\\
1 & 1 & 1 & 1 & 1 & 1\\
1 & 1 & 1 & 1 & 1 & 1\\
\frac{1}{4} & \frac{1}{3} & \frac{1}{2} & 1 & 1 & 1
\end{array}
\right]  .
\]
Let%
\[
q=\left[
\begin{array}
[c]{cccccc}%
4 & 3 & 4 & 2 & 2 & 1
\end{array}
\right]  ^{T}%
\]
and%
\[
p=\left[
\begin{array}
[c]{cccccc}%
3 & 4 & 4 & 2 & 2 & 2
\end{array}
\right]  ^{T}.
\]
The vectors $p$ and $q$ are efficient for $A.$ However, $p+q$ is not efficient
for $A,$ implying that $\mathcal{E}(A)$ is not convex.
\end{example}

A reciprocal matrix $A\in\mathcal{PC}_{n}$ that is obtained from a consistent
matrix by changing two entries in symmetrically placed positions is a
particular case of a column perturbed consistent matrix, and was called in
\cite{p6,CFF} \emph{a simple perturbed consistent matrix}. Thus, as a
consequence of Theorem \ref{t6}, $\mathcal{C}(A)\mathcal{\subseteq E}(A),$
implying that the Perron vector of $AA^{T}$ and of $A$ are efficient. This
gives a much simpler proof of the known efficiency of the Perron vector of $A$
\cite{p6,CFF}. We note that any $3$-by-$3$ reciprocal matrix is a simple
perturbed consistent matrix \cite{FJ1}.

A simple perturbed consistent matrix $A\in\mathcal{PC}_{n}$ is monomial
similar to a matrix, say $S_{n}(x_{1}),$ as in (\ref{AA}), with $x_{1}\geq1$
and $x_{2}=\cdots=x_{n-1}=1.$ The efficient vectors $w$ for $S_{n}(x_{1})$
were described in \cite{CFF}. Then, by Lemma \ref{lsim}, the efficient vectors
for $A$ may be obtained.

\begin{lemma}
\cite{CFF}\label{tmain} Let $n\geq3$ and $x\geq1.$ Let $w=\left[
\begin{array}
[c]{cccc}%
w_{1} & \cdots & w_{n-1} & w_{n}%
\end{array}
\right]  ^{T}$ be a positive $n$-vector. Then $w$ is efficient for
$S_{n}(x)\in\mathcal{PC}_{n}$ if and only if
\[
w_{n}\leq w_{i}\leq w_{1}\leq xw_{n},\text{ for }i=2,\ldots,n-1.
\]

\end{lemma}

\begin{theorem}
\label{t2}If $A\in\mathcal{PC}_{n}$ is a simple perturbed consistent matrix
then $\mathcal{E}(A)$ is convex.
\end{theorem}

\begin{proof}
Taking into account Corollary \ref{cconv}, we may assume that $A=S_{n}(x)$
with $x\geq1.$ Since a vector $w$ in $\mathcal{E}(A)$ is defined by a finite
set of linear inequalities on its (positive) entries, as described in Lemma
\ref{tmain}, it follows that $\mathcal{E}(A)$ is convex.
\end{proof}

\bigskip

The cone $\mathcal{C}(A)$ may be properly contained in $\mathcal{E}(A),$ even
when $\mathcal{E}(A)$ is convex, as we next illustrate.

\begin{example}
For the simple perturbed consistent matrix $S_{5}(4)\in\mathcal{PC}_{5},$ the
vector $\left[
\begin{array}
[c]{ccccc}%
4 & 2 & 3 & 4 & 1
\end{array}
\right]  ^{T}$ is efficient for $A$ and is not in $\mathcal{C}(A)$, as the
middle entries are distinct.
\end{example}

\section{Empirical efficiency of the singular vector and the right Perron
vector\label{s2}}

For each $n\leq25,$ we generate $n$-by-$n$ matrices $A_{i}\in\mathcal{PC}_{n}%
$, $i=1,\ldots10000$, as follows. We first generate an $n$-by-$n$ matrix
$C_{i}$ with entries from a uniform distribution in $(0.1,15)$ and then
construct the matrix $A_{i}$ as $C_{i}\circ C_{i}^{(-T)},$ in which $\circ$
denotes the Hadamard product and $C_{i}^{(-T)}$ denotes the transpose of the
entrywise inverse of $C_{i}.$ For each $n,$ we count the number of matrices
$A_{i}$ for which the Perron vectors of $A_{i}$ and of $A_{i}A_{i}^{T}$ are
efficient for $A_{i}.$ We also generate a random positive vector for each
$A_{i},$ with entries from a uniform distribution in $(0,5),$ and test it for
efficiency, as a baseline. The results are presented in Figure \ref{fig1}. We
can observe that, for $n$ close to $12$ and beyond, almost always the right
Perron vector and the left singular vector are efficient for $A_{i}.$ Not
surprisingly, the random vector is efficient much less often.
%

\begin{figure}[h]
 \includegraphics[width=\linewidth]{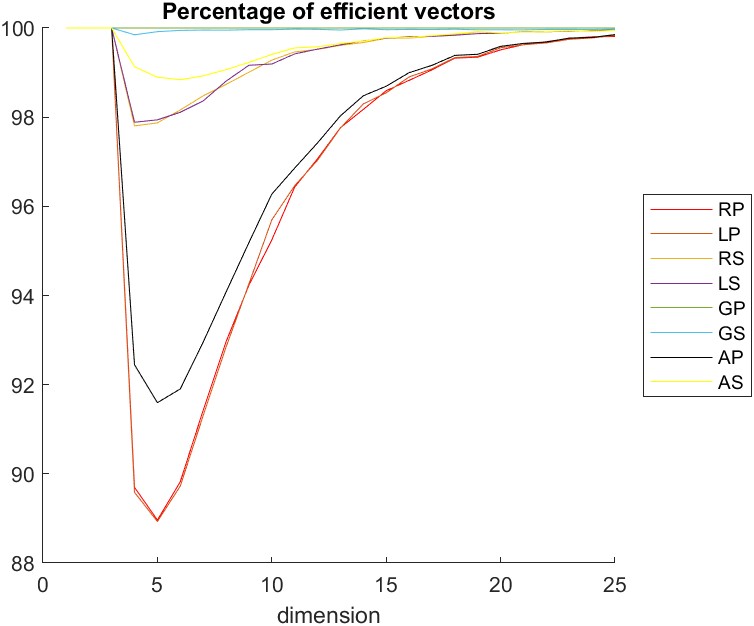}
 \caption{\label{fig1} Number of times the Perron vector, the singular vector and a random
vector are efficient, for $10000$ random reciprocal matrices of size $n,$ for
each $n\leq25.$}

\end{figure}

We then generate $10000$ random matrices $A_{i}\in\mathcal{PC}_{n}$, for each
$n=4,5,7,9,12,15,20,$ following the same procedure as before and obtain for
each matrix the Perron vector $v_{P}$ and the singular vector $v_{S}$. In
Table \ref{tab3} we compare the efficiency of these vectors for each matrix.
We can observe that, when just one of the vectors $v_{P}$ or $v_{S}$ is
efficient, $v_{S}$ is the one that is efficient more often.%

\begin{table}[] \centering
$%
\begin{tabular}
[c]{|c|c|c|c|c|c|c|c|}\hline
$n$ & $4$ & $5$ & $7$ & $9$ & $12$ & $15$ & $20$\\\hline
\multicolumn{1}{|l|}{$v_{P}$ is efficient and $v_{S}$ is inefficient} & $123$
& $137$ & $106$ & $80$ & $42$ & $22$ & $6$\\\hline
\multicolumn{1}{|l|}{$v_{P}$ is inefficient and $v_{S}$ is efficient} & $908$
& $901$ & $592$ & $341$ & $115$ & $41$ & $10$\\\hline
\multicolumn{1}{|l|}{$v_{P}$ and $v_{S}$ are efficient} & $8880$ & $8863$ &
$9251$ & $9548$ & $9841$ & $9935$ & $9983$\\\hline
\multicolumn{1}{|l|}{$v_{P}$ and $v_{S}$ are inefficient} & $89$ & $99$ & $51$
& $31$ & $2$ & $2$ & $1$\\\hline
\end{tabular}
\ \ \ $%
\caption{Comparison of the efficiency of the singular vector  $v_S$ and of the Perron vector $v_P$  for $10000$ random reciprocal matrices of size $n$. The number of matrices in each case is given.}\label{tab3}%
\end{table}%

\section{Conclusions\label{s6}$\allowbreak$}

Motivated by the fact that the normalized outer product of the left and right
singular vectors (associated with the spectral norm) of a matrix $A$ give the
best rank 1 approximation to $A$ (in the Frobenius norm), we developed theory
for the Perron vector of $AA^{T}$ for $A\in\mathcal{PC}_{n}$ and its efficient
vectors. This is compared to the right Perron vector for $A.$ These vectors
are efficient for a reciprocal matrix $A\in\mathcal{PC}_{n}$ that contains a
consistent $(n-1)$-by-$(n-1)$ principal submatrix, due to the shown fact that
any vector in the cone generated by the columns of $A$ is efficient for $A.$
Then, simulations, generating random matrices $A\in\mathcal{PC}_{n}$, test how
often each vector is efficient, as a function of $n.$ Typically, the Perron
vector of $AA^{T}$ performs better than the traditional right Perron vector of
$A,$ though both perform well for large $n.$ But as the singular vector is
nearly as easy to calculate, it should be preferred.

\bigskip

We declare there are no conflicts of interests.

\end{document}